\documentclass[twoside,a4paper,reqno,11pt]{amsart}
\usepackage{amsfonts, amsbsy, amsmath, amsthm, amssymb, latexsym, verbatim, enumerate}
\usepackage{mathrsfs}
\usepackage[top=30mm,right=30mm,bottom=30mm,left=30mm]{geometry}

\usepackage{hyperref}
\usepackage{amssymb,amsfonts,amsmath,amsopn,amstext,amscd,latexsym,amsthm}
\usepackage{bbm}
\usepackage[all,cmtip]{xy}

\usepackage{enumerate}
\usepackage[shortlabels]{enumitem}

\theoremstyle{plain}

\begin{document}

\def\a{\alpha}
 \def\b{\beta}
 \def\e{\epsilon}
 \def\d{\delta}
  \def\D{\Delta}
 \def\c{\chi}
 \def\k{\kappa}
 \def\g{\gamma}
 \def\Ind{\mathrm{Ind}}
 \def\t{\tau}
\def\ti{\tilde}
 \def\N{\mathbb N}
 \def\Q{\mathbb Q}
 \def\Z{\mathbb Z}
 \def\C{\mathbb C}
 \def\F{\mathbb F}
 \def\ovF{\overline\F}
 \def\bfN{\mathbf N}
 \def\cG{\mathcal G}
 \def\cT{\mathcal T}
 \def\cX{\mathcal X}
 \def\cY{\mathcal Y}
 \def\cC{\mathcal C}
 \def\cD{\mathcal D}
 \def\cZ{\mathcal Z}
 \def\cO{\mathcal O}
 \def\cW{\mathcal W}
 \def\cL{\mathcal L}
 \def\bfC{\mathbf C}
 \def\bfZ{\mathbf Z}
 \def\bfO{\mathbf O}
 \def\G{\Gamma}
 \def\bO{\boldsymbol{\Omega}}
 \def\bgo{\boldsymbol{\omega}}
 \def\go{\rightarrow}
 \def\do{\downarrow}
 \def\ra{\rangle}
 \def\la{\langle}
 \def\fix{{\rm fix}}
 \def\ind{{\rm ind}}
 \def\rfix{{\rm rfix}}
 \def\diam{{\rm diam}}
 \def\uni{{\rm uni}}
 \def\diag{{\rm diag}}
 \def\Irr{{\rm Irr}}
 \def\Syl{{\rm Syl}}
 \def\Out{{\rm Out}}
 \def\Tr{{\rm Tr}}
 \def\M{{\cal M}}
 \def\cE{{\mathcal E}}
\def\td{\tilde\delta}
\def\tx{\tilde\xi}
\def\DC{D^\circ}
\def\ext{{\rm Ext}}
\def\res{{\rm Res}}
\def\Ker{{\rm Ker}}
\def\hom{{\rm Hom}}
\def\End{{\rm End}}
 \def\rank{{\rm rank}}
 \def\soc{{\rm soc}}
 \def\Cl{{\rm Cl}}
 \def\A{{\sf A}}
 \def\sP{{\sf P}}
 \def\sQ{{\sf Q}}
 \def\SSS{{\sf S}}
  \def\SQ{{\SSS^2}}
 \def\St{{\sf {St}}}
 \def\p{\ell}
 \def\ps{\ell^*}
 \def\SC{{\rm sc}}
 \def\supp{{\sf{supp}}}
  \def\cR{{\mathcal R}}
 \newcommand{\tw}[1]{{}^#1}

\def\Der{{\rm Der}}
 \def\Sym{{\rm Sym}}
 \def\PSL{{\rm PSL}}
 \def\SL{{\rm SL}}
 \def\Sp{{\rm Sp}}
 \def\GL{{\rm GL}}
 \def\SU{{\rm SU}}
 \def\GU{{\rm GU}}
 \def\SO{{\rm SO}}
 \def\PO{{\rm P}\Omega}
 \def\Spin{{\rm Spin}}
 \def\PSp{{\rm PSp}}
 \def\PSU{{\rm PSU}}
 \def\PGL{{\rm PGL}}
 \def\PGU{{\rm PGU}}
 \def\Iso{{\rm Iso}}
 \def\Stab{{\rm Stab}}
 \def\GO{{\rm GO}}
 \def\Ext{{\rm Ext}}
 \def\E{{\cal E}}
 \def\l{\lambda}
 \def\ve{\varepsilon}
 \def\Lie{\rm Lie}
 \def\s{\sigma}
 \def\O{\Omega}
 \def\o{\omega}
 \def\ot{\otimes}
 \def\op{\oplus}
 \def\oc{\overline{\chi}}
 \def\pf{\noindent {\bf Proof.$\;$ }}
 \def\Proof{{\it Proof. }$\;\;$}
 \def\no{\noindent}
\def\hal{\unskip\nobreak\hfil\penalty50\hskip10pt\hbox{}\nobreak
 \hfill\vrule height 5pt width 6pt depth 1pt\par\vskip 2mm}

 \renewcommand{\thefootnote}{}

\newtheorem{theorem}{Theorem}
 \newtheorem{thm}{Theorem}[section]
 \newtheorem{prop}[thm]{Proposition}
 \newtheorem{conj}[thm]{Conjecture}
 \newtheorem{question}[thm]{Question}
 \newtheorem{lem}[thm]{Lemma}
 \newtheorem{lemma}[thm]{Lemma}
 \newtheorem{defn}[thm]{Definition}
 \newtheorem{cor}[thm]{Corollary}
 \newtheorem{coroll}[theorem]{Corollary}
\newtheorem*{corB}{Corollary}
 \newtheorem{rem}[thm]{Remark}
 \newtheorem{exa}[thm]{Example}
 \newtheorem{cla}[thm]{Claim}

\numberwithin{equation}{section}
\parskip 1mm

\title{Sectional rank and Cohomology}

\author[Guralnick]{Robert M. Guralnick}
\address{R.M. Guralnick, Department of Mathematics, University of Southern California, Los Angeles,
CA 90089-2532, USA}
\email{guralnic@usc.edu}
 
\author[Tiep]{Pham Huu Tiep}
\address{P. H. Tiep, Department of Mathematics, Rutgers University, Piscataway, NJ 08854, USA}
\email{tiep@math.rutgers.edu}

 
\date{\today}

\thanks{The first author was partially supported by the NSF
grants DMS-1600056.
The second author was partially supported by the NSF grant DMS-1840702.
The paper is partially based upon work 
supported by the NSF under grant DMS-1440140 while both authors were in residence at 
the Mathematical Sciences Research Institute in Berkeley, California, during the Spring 2018
semester. It is a pleasure to thank the Institute for support, hospitality, and stimulating environment.}

\thanks{The authors are grateful to Radha Kessar for bringing 
this question to their attention, and to her, Meinolf Geck, and Raphael Rouquier for 
helpful discussions.}

\begin{abstract}
Donovan's conjecture implies a bound on the dimensions of cohomology groups
in terms of the size of a Sylow $p$-subgroup and we give a proof of 
a stronger bound (in terms of sectional $p$-rank) for $\dim H^1(G,V)$.   We also prove a reduction
theorem for higher cohomology. 
\end{abstract}

\maketitle



\section{Introduction}

Let $G$ be a finite group, $p$ a prime and $k$ an algebraically closed field of characteristic $p$.
Donovan's conjecture (cf. \cite{K}) asserts that for a fixed $p$-group $D$, 
there are only finitely many
blocks $B$  of any group algebra $kG$ with defect group $D$ up to Morita equivalence.

A trivial consequence of this conjecture is that there is
a bound on the  dimension of $\Ext$-groups between irreducible
modules (depending only on the defect group of the block
containing the irreducibles). 

Our main result considers what happens for the projective cover of the trivial module $k$
and $H^1$ under a weaker condition, where we do not fix the (isomorphism type of) Sylow $p$-subgroups 
but only their sectional $p$-rank.  Recall that the {\it sectional $p$-rank} of a finite group $G$ is the maximal rank 
of an elementary abelian group isomorphic to $L/K$ for some subgroups $K \lhd L$ of $G$.  
Even considering the case that $G=P$ is cyclic, one sees that there is
no upper bound on the  composition length of the projective cover of $k$.     We do prove:

\begin{thm} \label{main} Let $G$ be a finite group, $p$ a prime and $k$ an algebraically closed
field of characteristic $p$.   Let $r$ be the sectional $p$-rank of $G$.   There exists a constant
$C=C(p,r)$ such that if $J$ is the radical of the projective cover of the trivial $G$-module $k$, then
$J/J^2$ is a direct sum of at most $C$ irreducible $kG$-modules. 
\end{thm}

We conjecture that the constant can be chosen to depend only on the sectional 
$p$-rank and not on the prime $p$. The proof we give shows that the only obstruction to proving this
is the case of simple groups.     We first prove a reduction to the case
of simple groups.   The sectional rank assumption implies that (for a fixed $p$)
aside from finitely many simple groups,  it suffices to consider cross characteristic modules 
of finite simple groups of Lie type and of bounded rank.
We then  use the main result of \cite{GT} which essentially proves
the theorem in that case.   The results of \cite{GT} shows that the constant $C$ can be chosen
to be $|W| + e$ where $W$ is the
Weyl group and $e$ is the twisted rank of $G$.   We improve
this result in Section \ref{sec:GT}. 

We also conjecture that this is true for all projective 
indecomposable modules for $G$ (assuming bounded sectional $p$-rank of the 
defect group of the block). 
Some evidence for this follows from results of Gruber and 
Hiss \cite{GHi} about classical
groups (but with restrictions on  the primes). 

There are also some related results of  Malle and Robinson \cite{MR} aimed towards proving
their conjecture that the number of simple modules in a given $p$-block is at most 
$p^r$ where $r$ is the sectional rank of the defect groups of the block.   One cannot hope to bound
this number independently of $p$. 

A restatement of Theorem \ref{main} is the following:

\begin{cor} \label{maincor}
Let $G$ be a finite group, $p$ a prime and $k$ an algebraically closed
field of characteristic $p$.   Let $r$ be the sectional $p$-rank of $G$.   
Then there exist constants $A(p,r)$ and $B(p,r)$ such that
\begin{enumerate}[\rm(i)]
\item the number of irreducible $kG$-modules $V$ with $H^1(G,V) \ne 0$
is at most $A(p,r)$; and
\item  if $V$ is an irreducible $kG$-module, then $\dim H^1(G,V) \le B(p,r)$.
\end{enumerate}
\end{cor}

If one works with indecomposable modules, it is easy to see,
using the Green correspondence, that 
the problem reduces to the case that the Sylow $p$-subgroups are normal.
However,  there are indecomposable $P$-modules $V$ with arbitrarily large
$\dim H^1(P,V)$ for most $p$-groups $P$. 
The one case where this does yield information is when 
$G$ has a cyclic Sylow $p$-subgroup  (i.e. the sectional $p$-rank is $1$).  It is well
known (using the Green correspondence) 
that $\dim H^n(G,V) \le 1$ for an {\it indecomposable} module $V$ in
characteristic $p$ (cf.  \cite[Lemma 3.5]{GKKL}).

In \cite{Gu1}, the first author asked whether there is a universal constant $C$ such that
$\dim H^1(G,V) < C$ for $V$ any {\it faithful} absolutely irreducible
$G$-module with $G$ a finite group.   This is still open but likely false
(see \cite{Lu} for examples with very large $\dim H^1(G,V)$).   The existence
of absolutely irreducible modules for simple groups with large first cohomology group
depends on the validity of Lusztig's conjecture and on knowing that certain coefficients of 
Kazhdan-Lusztig polynomials can be very large 
(this gives examples for groups of  Lie type and modules in the natural characteristic).   
In particular, there are no known examples in small characteristic. 

Of course, $\dim H^1(G,k)$ can be arbitrarily large but is bounded if the sectional rank
of the Sylow $p$-subgroups is bounded (indeed, it is bounded in terms of the 
Frattini quotient of a Sylow $p$-subroup).   Thus, one needs to assume faithfulness
or some condition on the Sylow $p$-subgroups
to get upper bounds.  For faithful absolutely irreducible modules, the upper bounds for
$\dim H^1$ reduce to the case of finite simple groups.   It is known \cite{CPS, GT} that for
$G$ a finite simple groups of Lie type of bounded rank $s$, there is a bound
$\dim H^1(G,V) < C(s)$ for $V$ any absolutely irreducible $kG$-module.  
 
 There is a recent paper \cite{EL} giving a reduction in the case of 
abelian defect groups and proving Donovan's conjecture when $p=2$ and
the defect group is abelian.   There are also reductions to quasisimple
groups \cite{Du} (in the case of nonabelian defect groups, the reduction is in terms
of Cartan matrices rather than Morita equivalence). 

Consider the following question, where $|G|_p$ denotes the $p$-part of the order $|G|$.

\begin{question} \label{cohn}   Does there exist a constant $C=C(r,n)$ such that 
$\dim  H^n(G,V) \le C$ for any finite group $G$  with $|G|_p \le p^r$
and  $V$ an irreducible $kG$-module ?
\end{question} 

One can ask whether there is such a bound in terms of sectional rank
(and perhaps the constant depends on $p$ as well).  
    Likely this can be reduced to
two questions.  The first is whether this holds for finite simple groups.
The second is whether there is a bound on $\dim H^n(P,k)$ for a $p$-group
$P$ (in terms of sectional rank).   There is a result of Quillen (see \cite{AE}
for a generalization to any module) 
showing that the growth rate of $H^n(P,k)$ is determined by the maximal rank
of an elementary abelian subgroup of $P$.

We do reduce Question \ref{cohn} to the case of simple groups.  

\begin{thm}  \label{red}   Let $k$ be an algebraically closed field of characteristic
$p$.   Let $G$ be a finite group with $|G|_p \leq p^r$.  If there exists 
a constant $C=C(p, r, n)$ such that $\dim H^n(S,V) \le C$ for every finite simple
group $S$ and $V$ any irreducible $kS$-module,  then the same is true for any finite group $G$ (with possibly a different constant). 
\end{thm}

One can also raise a similar question about $\mathrm{Ext}$.   There should be a reduction
to the case of simple groups.   We give an example showing that
$\dim \mathrm{Ext}^1_G(V,W)$ can be arbitrarily large even for $V, W$
absolutely irreducible faithful modules.   See \cite{GKKL} for a similar 
example for $H^2$.   

We also improve our $H^1$ results from \cite{GT} giving
bounds in terms of the sectional rank but also depending upon
the prime.   Here are some of the results in this direction:

\begin{thm} \label{GRresult}   Let $G$ be a finite simple group of Lie type in
characteristic $\ell$.  Let $W$ be the Weyl group of $G$.
 Let $p \ne \ell$ be a prime and $k$ an algebraically
closed field of characteristic $p$. Then the following statements hold.
\begin{enumerate}[\rm(i)]
\item The number of irreducible $kG$-modules $V$ with $H^1(G,V) \ne 0$
is at most $| \Irr(W)| + 3$.  
\item Suppose that $p \nmid [G_i:B]$ for any
minimal parabolic subgroup $G_i$ containing a fixed Borel subgroup $B$ of $G$. Then 
the number of irreducible $kG$-modules $V$ with $H^1(G,V) \ne 0$
is less than $| \Irr(W)|$.
\end{enumerate}
\end{thm}

\begin{thm} \label{GT1}  Let $G$ be a finite simple group of Lie type  in
characteristic $\ell$ of twisted rank $e$.   Let $W$ be the Weyl group of $G$.
 Let $p \ne \ell$ be a prime and $k$ an algebraically
closed field of characteristic $p$.   Let $V$ be an irreducible  
$kG$-module.   Let $G_1, \ldots, G_e$ denote the minimal
parabolic subgroups properly containing a fixed Borel subgroup $B$ of $G$.  
\begin{enumerate}[\rm(i)]
\item If $p \nmid |G_i|$ for $1 \leq i \leq e$, then
$\dim H^1(G,V) \le \dim V^B < |W|^{1/2}$. 
\item If $p \nmid |B|$, then $\dim H^1(G,V) < |W|^{1/2}$.
\item In general, $\dim H^1(G,V) < e-1+|W|^{1/2}$. If $V^B = 0$, then $\dim H^1(G,V) \leq 1$.
\item Moreover, if $V_1, \ldots,V_m$ are pairwise non-isomorphic representatives of isomorphism 
classes of irreducible $kG$-modules with $V^B \neq 0$ and $V \not\cong k$, then
$$\sum^m_{i=1}(\dim H^1(G,V_i)+1 -e/2)^2 \leq me^2/4 + |W|+e+1.$$ 
\end{enumerate}
\end{thm}

Theorem \ref{GT1}(iii) shows that the sum of squares of $\dim H^1(G,V)$, adjusted suitably, with $V$ running over 
all isomorphism classes of irreducible $kG$-modules, is also 
bounded roughly by $|W|$, and $m \leq |\Irr(W)|-1$ by Theorem \ref{GRresult}. See Section 4  for details and other related results.

\section{Sectional Rank and $H^1$}

Fix a prime $p$ and $k$ an algebraically closed field of characteristic $p$. 
If $G$ is a finite group, let $s(G)=s_p(G)$ be the sectional $p$-rank of $G$.   
If $V$ is a $kX$-module for a group $X$, we let $\bfC_X(V)$ be the kernel of the representation
sending $X$ to $\GL(V)$ and let $V^X$ denote the submodule of $X$-fixed points in $V$.

A key result is the following easy consequence of the main result of \cite{GT}.
Again, we conjecture that the constants can be chosen independently of $p$.

\begin{thm} \label{GT} Let $S$ be a finite nonabelian simple group with $s(S)=s$
fixed.   Then
there exist constants $A(p,s)$ and $B(p,s)$ such that:
\begin{enumerate}[\rm(i)] 
\item  if $V$ is an irreducible $kS$-module, then 
$\dim H^1(S,V) \le B(p,s)$; and 
\item  there are at most $A(p,s)$ irreducible $kS$-modules $V$
with $H^1(S,V) \ne 0$.
\end{enumerate}
\end{thm}

\begin{proof}   Excluding only finitely many simple groups (depending upon $s$ and $p$), 
we see that it reduces to the case that $S$ is a finite simple group of Lie type
in characteristic other than $p$.   The result now follows by \cite{GT} (see also
Section \ref{sec:GT} below for better results) where it was
shown that $\dim J/J^2 \le |W| + e$, where $J$ is the radical of the projective cover
of $k$, $G$ is a finite simple group of Lie type of twisted rank $e$ with Weyl group
$W$ and $p$ is not the characteristic of $G$.
\end{proof}

We first give a quick proof of Corollary \ref{maincor}(ii).

\begin{cor}  \label{h1bound}  There exists 
a constant $C(p,s)$ such that
$$\dim H^1(G,V) \le C(p,s)$$
for any finite group $G$ with $s_p(G)=s$ and any irreducible $kG$-module $V$.
\end{cor}

\begin{proof}  Let $Q:=\bfC_G(V)$.  By the inflation-restriction sequence in cohomology,
$$\dim H^1(G,V) \le \dim H^0(G/Q,H^1(Q,V)) + \dim H^0(G/Q, V).
$$
Since $Q$ acts trivially on $V$,   
$H^0(G/Q,H^1(Q,V)) \cong \hom_G(Q,V)
= \hom_G(Q/Q_1,V)$ where $Q/Q_1$ is the largest elementary abelian $p$-group
quotient of $Q$.  Since $|Q/Q_1| \le p^s$, we can view $Q/Q_1$ as an $\F_pG$-module of dimension
$\leq s$. As $V$ is irreducible, it follows that 
$\dim \hom_G(Q/Q_1,V) \le s/(\dim V)$.  So it suffices to assume that
$Q=1$ and $V$ is faithful.

In that case, it follows from \cite{Gu2} that $\dim H^1(G,V) \le \dim H^1(S,W)$
where $S$ is a subnormal simple subgroup of $G$ and $W$ is an $S$-submodule
of $V$ that is irreducible.   Now apply Theorem \ref{GT}.
\end{proof} 

We now turn towards the proof of Corollary \ref{maincor}(i).   We essentially
split the problem into two cases.   The first is when the module occurs
as a split chief factor in the group and the second is when
$H^1(G/\bfC_G(V), V) \ne 0$.  
Recall that a {\it split chief $p$-factor} of a finite group $G$ is a chief factor $H/K$ with
$K \lhd H$ and $H/K$ a $p$-group such that $H/K$ has a complement
in $G/K$.   

\begin{lemma} \label{lemma1}  Let $G$ be a finite group of sectional $p$-rank $s$.
In any minimal normal series of $G$,  there is a bound $D(s)$ 
on the number of split chief factors of $G$
that are $p$-groups. 
\end{lemma}

\begin{proof}  Consider a minimal counterexample.
  We may assume that the Frattini subgroup $\Phi(G)=1$, whence $\Phi(F^*(G))=1$.
We may also assume that $\bfO_{p'}(G)=1$. 

By induction, we may also assume that $E(G)=1$.   Thus,  
$F^*(G)$ is an elementary abelian $p$-group and is a semisimple
$G$-module.   Thus,   $|G| \le p^s |\GL_s(p)|$.  Thus, it suffices to
consider the problem for completely reducible subgroups of $\GL_s(p)$.
We just make the trivial observation that since the Sylow $p$-subgroup
of $G$ has order at most $p^{s(s+1)/2}$, the result is now clear.
\end{proof} 

Note that in Lemma \ref{lemma1}, we do need to consider split chief factors; indeed,
in a cyclic group of order $p^a$, the sectional rank is $1$ but the
number of chief factors is $a$.   If one only wanted a bound on the
number of $p$-chief factors up to $G$-isomorphism, the proof
above can be modified to obtain this (and this is all we need). 
We reiterate that the bound
above does not depend on $p$.   One could prove a much stronger
staement using results in \cite{GMP}.   

It is convenient to introduce $s_p'(G)$ which we define to be the maximal
sectional $p$-rank of a section $H/K$ of $G$ that is a direct product
of non-abelian simple groups.     

\begin{lemma} \label{lemma2}  Let $G$ be a finite group with $s':=s_p'(G)$.   
There exist constants $C_i(p,s')$ such that:
\begin{enumerate}[\rm(i)] 
\item  The   number of irreducible $kG$-modules $V$
such that $H^1(G/\bfC_G(V), V)  \ne 0$  is at most $C_1(p,s')$.
\item  $\dim H^1(G/\bfC_G(V),V) \le C_2(p,s')$.
\end{enumerate}
\end{lemma}

\begin{proof}   
Certainly, $s'_p(G/\bfC_G(V)) \leq s'$, so by induction on $|G|$ we may assume that $V$ is faithful.
Since $\bfO_p(G)$ acts trivially on any irreducible $kG$-module, 
we may assume that $\bfO_p(G) =1$.  If $\bfO_{p'}(G)$ acts nontrivially on $V$, 
then by the restriction-inflation sequence, we see that $H^1(G/\bfC_G(V),V)=0$.
Thus, we may also assume that $\bfO_{p'}(G)=1$. 

So $F^*(G) = S_1 \times \ldots \times S_t$ where the $S_i$'s are non-abelian simple
groups (and clearly $t \le s'$). 
By the above, $F^*(G)$ acts nontrivially 
on $V$, and so $H^0(F^*(G),V) = 0$. Decompose $V|_{F^*(G)} = c\oplus^d_{i=1}W_i$,
where the $W_i$ are $G$-conjugate, pairwise non-isomorphic irreducible $kF^*(G)$-modules and $c \geq 1$. Also write 
$W_i=W_{i,1} \otimes \ldots \otimes W_{i,t}$, where $W_{i,j}$ is a simple $S_j$-module.
If at least two of the $W_{1,j}$'s are nontrivial, then 
by the K\"unneth formula and the inflation-restriction sequence, $H^1(F^*(G),W_i) = 0$ and so 
$H^1(G/\bfC_G(V), V)=0$. Thus we may assume that $W_{1,1} \not\cong k$ but $W_{1,j} \cong k$ for
all $j > 1$. Now $G$ permutes the $S_j$'s. Assume this action is intransitive, say $S_t$ is not $G$-conjugate to $S_1$. 
Then the described shape of $W_1$ implies that $W_{i,t} \cong k$ for all $i$ and so $S_t \leq \bfC_G(V)$,
contrary to our assumption.  Hence, we may assume that $F^*(G)$ is the unique minimal normal subgroup of $G$
and all the $S_i$'s are $G$-conjugate.

The argument above shows that $H^1(G,V) \ne 0$ implies that 
$V=\Ind_{\bfN_G(S_1)}^G(W)$ for some irreducible $kN$-module $W$ with 
$N:=\bfN_G(S_1)$ (in fact, $W|_{F^*(G)} \cong cW_1$).   
Note that modding out by $\bfC_G(S_1) \geq S_2 \times \ldots S_t$ does not change the
computation for $H^1$.   Thus,  it suffices to consider the case that $F^*(G)=S$
is a simple group (with bounded $s_p'(S)$).   

Using Shapiro's Lemma once more, we may assume that $S$ acts homogeneously on $V$
and so (passing to a central $p'$-cover if necessary), we may assume that
$V = W \otimes U$ where $U$ is a $G/S$ module.  

Applying inflation-restriction sequence again, we see that
$$
H^1(G,V)= H^0(G/S, H^1(S,V)).
$$

By taking a $G$-resolution and restricting to $S$, we see
that $H^j(N,V) \cong H^j(S,W) \otimes U$ as a $G/S$-module. 
By Theorem \ref{GT}, this gives the bound on $\dim H^1$
and also shows there is a bound on the number of possible modules
$W$ so that $H^1(S,W) \ne 0$ (and so also $H^1(N,V) \ne 0$).
Thus, there are only finitely many simple modules of $N$ that we need to
consider and so we may fix this. 

Now $H^0(N/S, H^1(S,W) \otimes U))$ is nonzero if and only if $U$
is a quotient (as an $N/S$-module) of $H^1(S,W)$ and so there are
only finitely many possibilities for $U$, whence the result. 
\end{proof}

We can now prove Corollary \ref{maincor} (which is equivalent to 
Theorem \ref{main}).   

We have already shown in Corollary \ref{h1bound} that there is a bound on $\dim H^1(G,V)$.   
   
Next we show there is a bound on the number of irreducible 
$kG$-modules $V$ with $H^1(G,V) \ne 0$.   By Lemma \ref{lemma1}
there are only 
finitely many such modules which occur as split chief factors
of $G$ and so we may assume that $V$ is not a chief factor of $G$.
Thus, by \cite[2.10]{AG}, we have that $H^1(G,V)=H^1(G/\bfC_G(V), V)$
and then Lemma \ref{lemma2} applies.

\section{Higher Cohomology}

We fix a prime $p$ and an algebraically closed field $k$ of characteristic $p$. 

We first note the trivial result:

\begin{lemma} \label{trivial} Let $H$ be a finite group and $V=W_1 \otimes W_2$
a tensor product of $kH$-modules with $W_2$ irreducible.  Then 
$\dim H^0(H, W_1 \otimes W_2) \le (\dim W_1)/(\dim W_2)$.
\end{lemma}

We need the following result that follows from an easy spectral sequence
argument.  See \cite{Ho} or \cite[Lemma 3.7]{GKKL}.  

\begin{lemma} \label{holt}   Let $G$ be a finite group, $N$ a normal
subgroup of $G$ and $V$ a $kG$-module.  Then 
$\dim H^n(G,V) \le \sum_{i=0}^n  \dim H^i(G/N, H^{n-i}(N,V))$.
\end{lemma}

\begin{proof}[Proof of Theorem \ref{red}]
We induct on $n+ r$.
   If $r=0$ or $n=0$, then the result is clear.

Let $D_j$ be the maximum value for $\dim H^j(S,V)$ with
$S$ a non-abelian simple group with $|S|_p \le p^r$.
This exists by assumption.   

Let $V$ be an irreducible $kG$-module.   If $V \cong k$,
then the cohomology ring $H^*(G,k)$ embeds in 
$H^*(P,k)$ for $P$ a Sylow $p$-group and the result holds.
So assume that $V$ is nontrivial. 

By Shapiro's Lemma, we may assume that $V$ is a primitive $kG$-module
(otherwise $V=\Ind_H^G(W)$ for some proper subgroup $H$). 

Let $N$ be a maximal (proper) normal subgroup of $G$.  
Set $S=G/N$.   Then
$N$ acts homogeneously on $V$ by primitivity.  Passing to a $p'$-central
cover of $G$ if necessary, $V \cong W_1 \otimes W_2$
where $W_1$ is a $kG$-module that is $N$-irreducible and
$W_2$ is an irreducible $G/N$-module.  By Lemma \ref{holt}, 
$$
\dim H^n(G,V) \le \sum_{i=0}^n  \dim H^i(G/N, H^{n-i}(N,V)).
$$

As we observed earlier, we see
that $H^j(N,V) \cong H^j(N,W_1) \otimes W_2$ as a $G$-module.

If $p$ does not divide $|S|$, then we see by irreducibility of $W_2$ that
$$
\dim H^n(G,V) \le \dim H^0(G/N, H^n(N,W_1) \otimes W_2)
\le \dim H^n(N,W_1).
$$
Thus we may assume that $G/N$ has no nontrivial $p'$-quotients; in particular, $|N|_p < |G|_p$.
 
If $G/N$ has order $p$,  then $W_1 = V$ and so
$$\dim H^n(G,V) \le \sum_{i=0}^n   \dim H^i(N,V),$$
and this is at most $\sum_{j=0}^n C(p,r-1,j)$ and the result holds.  

More generally, if $G/N$ has order at most $e$, 
we can pass to $G_0$ where $G_0/N$ is a Sylow $p$-subgroup
of $G/N$.  Then the restriction map on cohomology from 
$G$ to $G_0$ is injective.   Note that $V$ restricted to $G_0$
has at most $\dim W_2 \leq e^{1/2}$ composition factors (all isomorphic to
$W_1$ as $N$-modules).   Using the previous case and induction,
we get a bound for $\dim H^n(G,V)$.   

The remaining case is when 
$S \cong G/N$ is a nonabelian simple group of sufficiently large order $e$.
As $|S|_p \leq p^r$ is bounded, we may choose $e$ sufficiently large so that 
$S$ is a simple group of Lie type in characteristic $\neq p$ and any nontrivial $S$-module of 
dimension less than the maximum dimension of the
possible dimensions  of $H^j(N,W_1)$ is trivial. 
Then
$$
\dim H^n(G,V) \le C_1\cdot \sum_{j=0}^n  \dim H^j(S, W_2),
$$
where $C_1$ is an upper bound for $\dim H^j(N,W_1)$
and the result follows. 
\end{proof}

\section{Cross Characteristic $H^1$} \label{sec:GT}
 
 In this section, we take $G$ to be a finite simple group of Lie type
 of twisted rank $e$ over the field of size $q$.  Fix a Borel subgroup $B$
 of $G$ with unipotent radical $Q$.    Let $G_1, \ldots, G_e$ denote
 the minimal parabolic subgroups properly containing $B$. 
    Let $p$ be a prime not dividing
 $q$ and $k$ an algebraically closed field of characteristic $p$.
 
 Our goal is improve the bounds from \cite{GT} on $\dim H^1(G,V)$ 
 with $V$ an irreducible $kG$-module.   
  
 We first prove Theorem \ref{GRresult} that improves the bound for the number of irreducible $kG$-modules with nontrivial
 $H^1$.  This critically depends on results of Geck and Rouquier (see \cite{GP})
 as well as results from \cite{GT}.  The original bound from \cite{GT} was of the magnitude  
 of $|W|$. 
 
 \begin{thm}  The number of irreducible $kG$-modules with nontrivial
 $H^1$ is at most $| \Irr(W)| + 3$.   If $p \nmid [G_i:B]$ for $1 \leq i \leq e$, 
 then this number is less than $| \Irr(W)|$.
 \end{thm}
 
 \begin{proof} It follows from results of Geck and Rouquier (see \cite[7.5.6, 8.2.5]{GP})
 that the number of distinct simple $kG$-modules $V$ with $V^B \ne 0$
 is at most  $|\Irr(W)|$.  By \cite[Theorem 1.3(ii)]{GT} and Corollary \ref{prime to index} (below), there are
 at most $4$ irreducible $kG$-modules with $V^B = 0 \ne H^1(G,V)$ and there are none
 if $p \nmid [G_i:B|$ for all $i$. Also note that $H^1(G,k)=0$ as $G$ is perfect. Hence both statements follow.
 \end{proof}
 
 Next we derive upper bounds on $\dim H^1(G,V)$.  Note that if $S$ is a simple $kG$-module, then, by Frobenius reciprocity, the 
 multiplicity of $S$ in the socle of $M = \Ind^G_B(k) = k_B^G$ is $\dim S^B = \dim S^Q$.
 In particular, this gives that 
 $$\sum_S (\dim S^B)^2  \le |W| = \dim \End_G(k_B^G),$$
 where the sum is over all (isomorphism classes of) simple $kG$-modules.  
 
First we record an elementary result.

\begin{lemma} \label{alperin-gor}  Let $G$ be a finite group.  Assume
that $G$ is generated by subgroups $H_1$ and $H_2$ and set $A:=H_1 \cap H_2$.
Let $V$ be a $kG$-module and assume that 
$$H^1(H_1,V)=H^1(H_2,V)=0.$$
Then $\dim H^1(G,V) \le \dim V^A$.
\end{lemma}

\begin{proof}  Let $D:=\Der(G,V)$ and consider the restriction map
$$\pi: D \to \Der(H_1,V) \times \Der(H_2, V).$$  
Since the $H_i$ generate $G$, $\pi$ is injective. For $\delta \in D$, let $\delta_i$ be the image
of $\delta$ in $\Der(H_i,V)$ with $i = 1,2$.   By assumption,  $\delta_i$ is the inner derivation $\delta(v_i)$ 
corresponding to some $v_i \in V$.  Since 
$\delta_1 - \delta_2$ vanishes on $A$, we see that $\dim \pi(D) \leq \dim V + \dim V^A$, 
whence the result.
\end{proof}

This has the following corollary.

\begin{cor} \label{prime to minimal}  Assume that $p \nmid |G_i|$
for all $i$.   If $V$ is any $kG$-module, then $\dim H^1(G,V) \le \dim V^B$.
If $V$ is irreducible, then $\dim H^1(G,V) < |W|^{1/2}$. 
\end{cor}

\begin{proof}  Split the set $\Delta$ of positive simple roots into two subsets $\Delta_j$, $j = 1,2$, so that the root subgroups 
in each subset commute (this is easy to do).  Let  $H_j$ be the subgroup of $G$ generated
by $B$ and the roots subgroups corresponding to $\Delta_j$ for $j = 1,2$.  
The construction of $\Delta_j$ and the assumption that $p \nmid |G_i|$ for all $i$ imply that
$p \nmid |H_j|$.   In particular, $H^1(H_j,V)=0$.
Clearly, $G=\langle H_1, H_2 \rangle$ as it contains all (positive simple) root subgroups, and $H_1 \cap H_2 = B$.   
Now the first statement follows by applying Lemma \ref{alperin-gor}.  The second statement also follows, since $\dim V^B < |W|^{1/2}$
by Frobenius reciprocity. 
\end{proof}

We generalize the previous results. 

\begin{lemma} \label{weak ag} Let $G$ be a finite group.  Assume
that $G$ is generated by subgroups $H_1$ and $H_2$ and set $B:=H_1 \cap H_2$.
Let $V$ be a $kG$-module.   Assume that the restriction maps from $H^1(H_i, V)$
to $H^1(B,V)$ are injective.  
Then $\dim H^1(G,V) \le \dim H^1(B,V) + \dim V^B$.
\end{lemma}

\begin{proof} This is very similar to the proof of Lemma \ref{alperin-gor}.  Let $\delta \in \Der(G,V)$
and let $\delta_i$ be the restriction of $\delta$ to $H_i$.  

Consider the restriction map $\pi_1$ from $\Der(G,V) \to \Der(H_1,V)$.   
Since $G$ is generated by $H_1$ and $H_2$, $\Ker(\pi_1)$ embeds
into $\Der_B(H_2,V)$, the space of the derivations on $H_2$ that are $0$ on $B$.
Now, if $\delta \in \Der_B(H_2,V)$, then $\delta$ is inner on $B$ and so also 
on $H_2$ (since the restriction map is injective on $H^1$).   Thus, 
$\Der_B(H_2,V)$ can be identified with $V^B$, and the result follows. 
\end{proof}

This gives:

\begin{cor} \label{prime to index}  Let $V$ be a $kG$-module.
 Assume that $p \nmid |G_i:B|$
for all $1 \leq i \leq e$.   If $V$ any $kG$-module, then the following statements hold.
\begin{enumerate}[\rm(i)] 
\item $\dim H^1(G,V) \le   \dim H^1(B,V) + \dim V^B$.
\item If $V$ is irreducible, then $\dim H^1(G,V) \le  (e +1)\dim V^B < (e+1)|W|^{1/2}$.
\item If $V$ is any $kG$-module,  $\dim H^1(G,V) \le  (e +1)\dim V^Q$.   
\item If $V$ is a submodule of $k_B^G$, then $\dim H^1(G,V) \le  (e +1)\dim V^B$.
\end{enumerate}
\end{cor}

\begin{proof}  
Let $H_1$ and $H_2$ be constructed as in the proof of Corollary 
\ref{prime to minimal}. Then the $H_j$, $j = 1,2$, are parabolic subgroups and $p \nmid [H_j:B]$.  Now apply
Lemma \ref{weak ag} to see that (i) holds. 

We next prove (ii) and so assume that $V$ is irreducible. 
If $V^B = 0$, then $H^1(G,V) = 0$ by \cite[Theorem 6.1]{GT}, and we are done.
So we may assume that $V^B \neq 0$.

Setting $R:=\bfO_{p'}(B)$, we have 
$V = [R,V] \oplus V^R$ and $H^1(B,[R,V])=0$.  Also, $V^R=V^B$ by \cite[Proposition 3.1]{GT}, and
so $H^1(B,V)=H^1(B,V^R) = H^1(B,V^B)$. As $B/R$ has rank $\leq e$ as an abelian group,
$\dim H^1(B,V^B) \leq e \dim V^B$ (and is in fact $0$ if $p \nmid |B|$).  
As noted previously,   $\dim V^B < |W|^{1/2}$, whence (ii) follows.   

Now (iii) follows from (ii) by the long exact sequence in cohomology  since $\dim V^Q$ is additive over composition factors.
Finally (iv) follows from (iii), since $\dim V^B = \dim V^Q$ for any submodule $V$ of $k_B^G$.
\end{proof}

So we have obtained an upper bound of the magnitude of $|W|^{1/2}$ unless
$p$ divides $[G_i:B]$ for some $G_i$.  There cannot be a result
in general that bounds $\dim H^1(G,V)$ in terms of $\dim V^B$
since there are (albeit very few) examples with $V^B=0$
and $\dim H^1(G,V)=1$, see \cite[\S6]{GT}.

We will give another bound in all cases, using the
property that 
\begin{equation}\label{h1-dual}
  \dim H^1(G,V)=\dim H^1(G,V^*)
\end{equation}    
for any irreducible $kG$-module $V$ {\it with $V^B \ne 0$} 
(in cross characteristic, and $G$ is a finite simple group of Lie type as before). 
For classical groups, any irreducible module $V$ is quasi-equivalent to its dual 
\cite[2.1, 2.4]{DGPS} and \cite{TZ} (i.e. $V^*$ is a twist of $V$ by an automorphism) whence
\eqref{h1-dual} holds (without the extra assumption that $V^B \ne 0$). 
The equality \eqref{h1-dual} for exceptional groups of Lie type follows from 
the following results about the socle and the head of indecomposable
 summands of $M=k_B^G$.  
 
 \begin{lemma}\label{socle}  Let $M=k_B^G$.
 \begin{enumerate}[\rm(i)]
 \item  $M$ is a direct sum of indecomposable modules with simple socle and
 simple head which are isomorphic.
 \item If $Y$ is an indecomposable summand of $M$, then the isomorphism
 class of $Y$ is determined by its socle (or head).
 \item If $Y$ is an indecomposable summand of $M$, then its socle and head
 are self-dual $kG$-modules.
 \end{enumerate}
 \end{lemma}
 
 \begin{proof}  (i) and (ii)  follow by \cite[1.20, 1.25, 1.28]{CE}. 
 Next we prove (iii).   Let $Y$ be an indecomposable summand of $M$
 with simple socle $S$.   We claim that $Y \cong Y^*$.  If we prove this,
 then by (i) and (ii), $S \cong S^*$ and the result follows.  Note that
 since $M$ is self-dual,  $Y^*$ is a summand of $M$ as well. 
 
Let $E:=\End_G(M)$ and let $X$ be the projective indecomposable $E$-module given by the Morita
correspondence (i.e. $X=\hom_G (Y, M)$ and $X^*$ the corresponding
projective module).   Let $\mathcal{O}$ be a discrete valuation ring in characteristic
$0$ with residue field contained in $k$ and let $K$ be the quotient field
of $\mathcal{O}$.   Let $M' := \Ind^G_B(\mathcal(O))$ be the corresponding induced module over $\mathcal{O}$
and let $E' := \End_G(M')$ be the corresponding endomorphism ring.   Then there is a bijection
between the indecomposable projective summands of $E$ and $E'$. 

Then $L:=\End_{KG}(\Ind_B^G(K)) \cong K \otimes E'$ is a Hecke algebra
and by a result of Lusztig (see \cite[8.4.7, 9.3.9]{GP})
 is split semisimple over $\mathbb{Q}[x^{1/2},x^{-1/2}]$, whence its projective modules
have real characters.  The same is true for $E'$ and so also for the Brauer characters for 
the projective indecomposables for $E$.  Thus, $X \cong X^*$ and $Y \cong Y^*$,
whence the result follows. 
\end{proof}
 
Note that if $p \nmid |B|$, then $M$ is projective and the first two statements of Lemma \ref{socle}
 hold trivially (since they hold for any projective indecomposable $kG$-module).   The self-duality
 statement \ref{socle}(iii) critically  requires Lusztig's result.   
 

\begin{cor}\label{mult1}
Let $V$ be an irreducible $kG$-module with $V^B \ne 0$.  
\begin{enumerate}[\rm(i)]
\item Then $V$ has multiplicity $\dim V^B$ in the socle of $k_B^G$, 
and $V \cong V^*$.
\item Let $V_2$ be an irreducible $kG$-module with $V_2^B \neq 0$ and $V_2 \not\cong V$. Suppose that $L$ is a $kG$-module 
with $\soc(L) = V_1 \cong V$ and $L/V_1 \cong V_2^{\oplus m}$ for some $m \geq 0$, and that 
$L \cong V_1 \oplus V_2^{\oplus m}$ as $B$-module. Let $X_i$ denote an indecomposable direct summand of $k^G_B$
with socle $V_i$ for $i = 1,2$. Then $L$ embeds in $X_1$.
\end{enumerate}
\end{cor}
 
\begin{proof}   
(i) The first statement follows by Frobenius reciprocity. Next, $V$ embeds in an indecomposable summand $Y$ of $k^G_B$.
By Lemma \ref{socle}(i) and (iii), we now have $V \cong \soc(Y)$ and $V \cong V^*$. 

\smallskip
(ii) Let $f_i := \dim V_i^B > 0$ for $i = 1,2$. Then  
$$\dim \hom_G(L,k^G_B) = \dim \hom_B(L,k) = f_1+mf_2,$$
whereas 
$$\dim \hom_G(L/V_1,k^G_B) = \dim \hom_B(L/V_1,k) = mf_2.$$
Now we apply Lemma \ref{socle} to decompose $k^G_B$ into a direct sum of its indecomposable direct summands. Let $Y$ be such 
a summand with $\soc(Y) = W$. Note that $\hom_G(L,Y) = 0$ if $W \not\cong V_1,V_2$. (Otherwise a nonzero quotient $L'$ of $L$ embeds in 
$Y$, and so either $V_1$ or $V_2$ embeds in $\soc(Y) = W$.) A similar argument shows that
$$\dim \hom_G(L,X_2) = \dim \hom_G(L/V_1,X_2) = m \cdot \dim \hom_G(V_2,X_2) = m.$$
As $X_i$ has multiplicity $f_i$ in $k^G_B$, it follows that 
$$\begin{aligned}
f_1 +mf_2 & = \dim \hom_G(L,k^G_B) = f_1 \cdot \dim \hom_G(L,X_1) + f_2 \cdot \dim \hom_G(L,X_2)\\
    & = f_1\cdot \dim \hom_G(L,X_1)+mf_2,
\end{aligned}$$    
and so $\dim \hom_G(L,X_1) = 1$. Also note that $\dim \hom_G(L/V_1,X_1) = 0$ as $\soc(X_1) \not\cong V_2$. Hence
$L$ embeds in $X_1$, as stated. 
\end{proof} 

Next we need to relate $\dim H^1(G,V)$ with the multiplicity of $V$  not in the socle of $M = k_B^G$
but in $\soc(M/k)$.   Note that if $p$ does not divide $B$, the projective cover $P(k)$  of
the trivial module is a direct summand of $k_B^G$ and so the multiplicity of any irreducible module
$V$ in $P(k)/k$ is precisely $\dim \mathrm{Ext}_G^1(V,k) = \dim H^1(G,V^*)$.  

We start by computing a related quantity.

\begin{lemma} \label{res1}  Let $V$ be an irreducible $kG$-module with $\dim V^B=f > 0$ and $V \not\cong k$, and set
$h := \dim H^1(G,V)$.
Let $a$ be the dimension of the image of $\res^G_B:H^1(G,V) \to H^1(B,V)$ in $H^1(B,V)$.   
Let $X$ be an indecomposable summand of $k_B^G$ with socle $V$, and 
let $J$ be the indecomposable summand of $k^G_B$ with trivial socle. 
\begin{enumerate}[\rm (i)]
\item  $\dim (X/V)^G = h-a$ and $a \le e/f$.
\item  The image of $\res_B^G:\ext^1_G(V,k) \to \ext^1_B(V,k)$ has dimension $a$ in $\ext^1_B(V,k)$.
\item  There exists a submodule $N$ of $J$ with $N/k$ a direct sum of $h-a$ copies of $V$.
\end{enumerate}
\end{lemma} 

\begin{proof}  
(a) Let $D:=\Der(G,V)$.  Then $D$ is the (unique) module with socle $V$ and trivial head of
dimension $h$.  By the definition of $a$, there is a subspace $D_0$ with $V \subseteq D_0 \subseteq D$, 
$\dim D/D_0 = a$,  such that $D_0 \cong V \oplus k^{\oplus (h-a)}$ as $B$-modules.  Of course
$V$ is still the socle of $D_0$. By Corollary \ref{mult1}(ii), $D_0$ embeds in $X$. We may then identify
$D_0$ with a submodule of $X$, and $\soc(D_0)$ with $\soc(X) = V$, and then have 
$$\dim (X/V)^G   \ge \dim (D_0/V)^G = h-a.$$   
Conversely, if $\dim (X/V)^G = h - b$, then there exists a submodule $Y \subseteq X$ with socle $V$ and $Y/V \cong k^{\oplus (h-b)}$.  
Since $V,Y \subseteq X \subseteq k_B^G$, we know by \cite[Proposition 3.1(ii)]{GT} that 
$$\dim Y^B = \dim Y^Q = (h-b)+\dim V^Q = (h-b) + \dim V^B = (h-b)+f.$$   
Next, the $B$-module $Y$ decomposes as $[Y,Q] \oplus Y^Q$, and likewise
$V = [V,Q] \oplus V^Q$. Counting the dimensions, we see that $[V,Q] = [Y,Q]$, and thus 
$$Y = [V,Q] \oplus Y^Q,$$
as $B$-module, with $B$ acting trivially on $Y^Q \supseteq V^Q$. We have therefore shown 
that $Y$ splits as $V \oplus k^{\oplus (h-b)}$ as a $B$-module.  Since $Y$
embeds in $D$, this implies by the definition of $a$ that  $h-b \leq h-a$, whence $\dim (X/V)^G = h-a$ as stated.

By Corollary \ref{mult1}(i), $V$ occurs in the socle of $k^G_B$ with multiplicity $f$. It follows by Lemma \ref{socle} 
that $X$ occurs as a direct summand of $k_B^G$ with multiplicity 
$f$. Hence,
$$f \cdot \dim H^1(G,X) \leq \dim H^1(G,k_B^G) = \dim H^1(B,k) \le e,$$
the latter inequality because the abelian group $B/Q$ has rank $\leq e$. We have shown that 
\begin{equation}\label{h11}
  \dim H^1(G,X) \le e/f.
\end{equation}

\smallskip
(b) We again look at the above constructed submodule $D_0$ of $X$, and consider the short exact sequence 
$$0 \to D_0 \to X \rightarrow X/D_0 \to 0.$$   
This gives rise to the sequence
\begin{equation}\label{h12}
  0 \to H^0(G, X/D_0) \to H^1(G,D_0) \to H^1(G,X).
\end{equation}  
Recall that $D_0/V \cong k^{\oplus (h-a)}$ and $\dim (X/V)^G  = h-a$ as shown in (a).  Together with $H^1(G,k) = 0$, this
implies that $H^0(G, X/D_0)=0$. Using \eqref{h11} and \eqref{h12}, we now see that 
\begin{equation}\label{h13}
  \dim H^1(G,D_0) \leq \dim H^1(G,X) \leq e/f.
\end{equation}  

We  now claim  
\begin{equation}\label{h14}
 \dim H^1(G,D_0)=a 
\end{equation} 
Consider the short exact sequence $0 \rightarrow  V \rightarrow D_0 \rightarrow  k^{\oplus (h-a)} \rightarrow 0$.
Thus, we have
$$0 = H^0(G,D_0) \to H^0(G,k^{\oplus (h-a)})   \to H^1(G,V) \to H^1(G,D_0) \to H^1(G,k^{\oplus (h-a)})=0,$$ 
and the claim follows. Thus, $a \le e/f$ as stated in (i).

\smallskip
(c) Note that the natural isomorphism from $H^1(G,V) = \ext^1_G(k,V)$ to
$\ext^1_G(V^*,k)$ gives an isomorphism of the subspace of each which are trivial on $B$, since if a short exact sequence
splits for $B$, so does its dual. It follows that the image of 
$$\res_B^G:\ext^1_G(V^*,k) \to \ext^1_B(V^*,k)$$
has dimension $a = h-\dim (X/V)^G$.  Now (ii) follows since $V \cong V^*$ by Corollary \ref{mult1}(i).


\smallskip
(d) By (ii), there exists a $G$-module $N$ with socle $k$ and $N/k \cong V^{\oplus (h-a)}$, such that 
$N \cong k \oplus V^{\oplus (h-a)}$ as $B$-module. 
By Corollary \ref{mult1}(ii), $N$ embeds in $J$.
\end{proof}


\begin{proof}[Proof of Theorem \ref{GT1}]  (i) is established in Corollary \ref{prime to minimal}.   We now prove (iii) and (iv).
By \cite[Corollary 6.5]{GT}, it suffices to consider $V_1, \ldots ,V_m$, pairwise non-isomorphic representatives of 
isomorphism classes of irreducible $kG$-modules $V$ with $V^B \neq 0$ and $V \not\cong k$. Note that $H^1(G,k) = 0$.

Keep the notation as in Lemma \ref{res1}, but with the index $i$ attached to the objects defined for $V_i$.
So $V_i$ is an irreducible $kG$-module, with $\dim V_i^B = f_i > 0$, $h_i = \dim H^1(G,V_i)$,
$J$ is the indecomposable summand of $k_B^G$ with trivial socle, and $N_i$ is a submodule of $J$ with 
$N_i/k \cong V_i^{\oplus (h_i-a_i)}$, and $a_i \leq e/f_i$. Working in $J/k$, we obtain a $G$-submodule $N \supseteq k = \soc(J)$
with
$$N/k \cong \oplus^m_{i=1}V_i^{\oplus (h_i-a_i)}.$$ 

Clearly, $\dim H^1(G,N/k)=\sum_ih_i(h_i-a_i)$.  By considering 
$$0 \to k \rightarrow N \to N/k \to 0,$$
we have 
$$0 = H^1(G,k) \to H^1(G,N) \to H^1(G,N/k) \to H^2(G,k).$$  
Let $\kappa := \dim H^2(G,k)$.
Note that $\kappa = 0$ if $p$ does not divide $|B|$. If $p||B|$, then $\kappa \le 1$ unless $p=2$ and $G$ is of type $D_m$
with $m$ even in which case $\kappa=2$.   Thus,  
\begin{equation}\label{h21}
  \dim H^1(G,N) \ge \sum_ih_i(h_i-a_i) - \kappa.
\end{equation}  
On the other hand, using
$$0 \to N \rightarrow J \rightarrow J/N \to 0,$$ 
we see that 
\begin{equation}\label{h22}
  \dim H^1(G,N) \leq \dim H^1(G,J) + \dim H^0(G, J/N).
\end{equation}  
We do not have a very good control over the last term. But note that 
$$\dim H^0(G, J/N) \leq \dim H^0(Q,J/N)  = \dim H^0(Q,J) - \dim H^0(Q,N).$$
By \cite[Proposition 3.1(ii)]{GT}, $H^0(Q, k_B^G) = H^0(B,k_B^G)$ has dimension $|W|$.   
Let $X_i$ denote an indecomposable summand of $k^G_B$ with socle $V_i$.
We have seen in the proof of Lemma \ref{res1} that $\dim H^0(Q,X_i) \ge f_i + (h_i-a_i)$
and $X_i$ occurs with multiplicity $f_i$ as a summand of $k_B^G$. As $V_i \not\cong k$, we get 
$$\dim H^0(Q,J) \le |W| - \sum_i(f_i^2+f_i(h_i-a_i)).$$
On the other hand, $\dim H^0(Q,N) = 1+\sum_if_i(h_i-a_i)$, 
and so 
$$\dim H^0(Q,J/N) \leq |W| - 1-\sum_i(f_i^2 +2f_i(h_i-a_i)).$$
Also we have that 
$$\dim H^1(G,J) + \sum_if_i\cdot\dim H^1(G,X_i) \leq \dim H^1(G,k_B^G) = \dim H^1(B,k) \le e,$$
and $\dim H^1(G,X_i) \geq \dim H^1(G,D_{0,i}) = a_i$ by \eqref{h13} and \eqref{h14}.  
Hence
$$\dim H^1(G,J) \le e - \sum_ia_if_i.$$
Putting all this in \eqref{h21} and \eqref{h22} yields:
$$\sum_ih_i(h_i-a_i) \le (e-\sum_ia_if_i) + |W|+\kappa-1 -\sum_i(f_i^2 +2f_i(h_i-a_i)).$$
Equivalently,
\begin{equation}\label{h23}
  \sum_i(h_i+f_i)(h_i+f_i-a_i) \leq |W|+e+\kappa-1.
\end{equation}  
In particular, for any $i$ we have 
$$(h_i+f_i)(h_i+f_i-a_i) \leq |W|+e+\kappa-1.$$
As $a_i \leq e/f_i$ by Lemma \ref{res1} and $f_i \geq 1$, we obtain for each $i$ that
$$h_i+f_i \leq \frac{a_i}{2} + \sqrt{\frac{a_i^2}{4} + |W|+e+\kappa-1} \leq  \frac{e}{2} + \sqrt{\frac{e^2}{4} + |W|+e+ \kappa-1} < e +|W|^{1/2}.$$
Thus $h_i < e+|W|^{1/2}-1$, as stated in (iii).

In general, 
$$(h_i+f_i)(h_i+f_i-a_i) \geq (h_i+f_i)(h_i+f_i-e/f_i) \geq (h_i+1)(h_i+1-e),$$
and so \eqref{h23} implies (iv).
 \end{proof}

A much simpler version of the previous proof gives a slightly better bound if $p \nmid |B|$. 
In that case  $H^1(B,V)=0=H^1(B,k)$ and $\kappa = 0$.    This yields Theorem \ref{GT1}(ii):
  
\begin{thm} \label{h1prime2borel}  If $p$ does not divide $|B|$,  then 
$\dim H^1(G,V) < |W|^{1/2}$.
\end{thm}

 We point out some easy corollaries.
 
\begin{cor}\label{sub1}  
Let $L$ be a $kG$-submodule of $k_B^G$. Then  $\dim H^1(G,L) < |W| + \dim H^1(B,k)$.
\end{cor}

\begin{proof}  This follows from the long exact sequence in cohomology applies to 
$$0 \rightarrow L \rightarrow k_B^G \rightarrow X \rightarrow 0.$$  
If $L^G \neq 0$, then this gives
$\dim H^1(G,L) \le \dim H^0(G,X) + \dim H^1(G,k_B^G)$, and the result holds since 
$$\dim H^0(G,X) \le \dim H^0(Q,X) < \dim H^0(Q,k^G_B) = |W|.$$   
If $L^G=0$, then replace $k_B^G$ by the sum $Z$ of all indecomposable summands of $k_B^G$ not containing the $G$-fixed space and
argue similarly (noting that $\dim H^1(G,Z) \leq \dim H^1(G,k_B^G)=\dim H^1(B,k)$).
\end{proof} 

If we assume that $p$ does not $|G_i|$ for any $i$, we can
get some stronger results.

\begin{cor}\label{sub2}  
Assume that $p \nmid |G_i|$ for all $i$. Let $L = X/Y$ with $Y < X$ $kG$-submodules of
$k_B^G$.  Then $\dim H^1(G,L) \le \dim L^B = \dim X^B - \dim Y^B  \le |W|$.   Moreover, $\dim H^1(G,X) \le |W|/2$.
\end{cor}

\begin{proof} The first statement follows by Corollary \ref{prime to index}(i) since $p \nmid |B|$.    
Now suppose that $X$ is a $kG$-submodule.   Let $M=k_B^G$ as above.  Arguing as in the proof of Corollary \ref{sub1}, we see
that $\dim H^0(G, M/X) = \dim H^1(G,X)$.   We also have that 
$$\dim H^1(G,X) \le \dim X^B = \dim X^Q.$$
Thus, 
$$\begin{aligned}
    \dim H^1(G,X) & = (1/2)(\dim H^1(G,X) + \dim H^0(G,M/X))\\ 
   & \le (1/2) (\dim X^Q + \dim (M/X)^Q) \\
   & = (1/2) \dim M^Q = |W|/2.
\end{aligned}$$
 \end{proof}  

This allows us to say something about $H^2$ (but only for submodules of $k_B^G$).

\begin{cor}\label{sub3} 
Assume that $p \nmid |G_i|$ for all $i$.   Let $L$ be a $kG$-submodule of 
$k_B^G$.   Then $\dim H^2(G,L)  \leq |W| - \dim L^B$.
\end{cor}

\begin{proof}   Consider the short exact sequence $0 \rightarrow L \rightarrow k_B^G \rightarrow X \rightarrow 0$.
This yields 
$$0 = H^1(G,k_B^G)  \rightarrow H^1(G,X) \rightarrow H^2(G,L) \rightarrow H^2(G,k_B^G)=0.$$
Thus, $H^2(G,L) \cong H^1(G,X)$ and the previous corollary applies. 
\end{proof}

We can do a bit better for irreducible $kG$-modules with nontrivial $B$-fixed points.

\begin{cor}\label{sub4} 
Assume that $p \nmid |G_i|$ for all $i$. Let $\mathcal{X}$  be the set of isomorphism classes of  
the  nontrivial irreducible $kG$-modules and let $f_V :=\dim V^G$.  Then
$$ \sum_{V \in \mathcal{X}} f_V \cdot \dim H^2(G,V) \leq |W| - \sum_{V \in \mathcal{X}} f_V^2.$$
\end{cor}  

\begin{proof}  Let $L$ be the complement to $k$ in the socle of $k_B^G$.  Then $L$
is the direct sum of $f_V$ copies of each $V \in \mathcal{X}$.   Now apply Corollary \ref{sub3},
noting that $\dim L^B = \sum_{V \in \mathcal{X}} f_V^2$.
\end{proof}

In particular, this implies that if $f := \dim V^B > 0$ and $V$ is an irreducible 
$kG$-module, then $\dim H^2(G,V) < |W|/f$. 

One can weaken the assumption that $p$ does not divide $|G_i|$ and obtain
some weaker results.  Unfortunately, these results do not yield any information
about modules with no $B$-fixed points.

 \section{An Example}
 
 Here we give an easy example showing that one cannot in general bound
 $\dim \mathrm{Ext}^1_G(V,W)$ for $V,W$ faithful absolutely irreducible $G$-modules.
 There are examples known as well using Kazhdan-Lusztig polynomials for 
 $G$ a simple finite group of Lie type and $V, W$ modules in the defining characteristic.
 There are no such examples known in for cross characteristic modules.  We give
 a trivial example for semisimple groups.
 
 Let $G=S_1 \times \ldots \times S_t$ be a direct product of $t$ finite non-abelian
 simple groups. Let $V_i$ be an absolutely irreducible $S_i$-module with
 $\dim \mathrm{Ext}_{S_i}^1(V_i, V_i)=e_i $.  
 Let $V=V_1 \otimes \ldots \otimes V_t$.   Then by the K\"unneth formula, we
 see that $\dim \mathrm{Ext}^1_G (V,V)= \sum_{i=1}^t e_i$.   Since there are
 examples with $e_i > 0$, we see that $\dim \mathrm{Ext}^1_G (V,V)$ can grow arbitrarily large with
 $t$ (but if the sectional rank of the Sylow $p$-groups is bounded, then so is $t$).

\end{document}